\documentclass[12pt]{amsart}
 \usepackage{amsaddr}
\usepackage{amsmath}
\usepackage{graphicx}
\usepackage{enumerate}
\usepackage{natbib}
\usepackage{amsthm}
\usepackage{amsfonts}%
\usepackage{amssymb}%
\usepackage{graphicx}
\usepackage{natbib}
\usepackage{booktabs}
\usepackage[font=small,labelsep=none]{caption}
\usepackage{url}

\newcommand{\bse}{\boldsymbol{\eta}}
\newcommand{\bsn}{\boldsymbol{\nu_k}}
\newcommand{\bsZ}{{\mathbf Z}}
\newcommand{\bsG}{{\mathbf G}}
\newcommand{\bsR}{{\mathbf R}}
\newcommand{\bsX}{{\mathbf X}}
\newcommand{\bsY}{{\mathbf Y}}

\newcommand{\bsB}{\boldsymbol{\beta}}
\newcommand{\bds}[1]{\boldsymbol{#1}}

\newcommand{\Var}{\mbox{var} }
\newcommand{\Tr}{\mbox{tr} }

\newcommand{\tr}{\mbox{tr} }

\newcommand{\ba}{{\mathbf a}}

\newcommand{\bc}{{\mathbf c}}

\newcommand{\be}{{\mathbf e}}

\newcommand{\bk}{{\mathbf k}}

\newcommand{\bt}{{\mathbf t}}

\newcommand{\bA}{{\mathbf A}}

\newcommand{\bE}{{\mathbf E}}

\newcommand{\bG}{{\mathbf G}}

\newcommand{\bI}{{\mathbf I}}

\newcommand{\bL}{{\mathbf L}}
\newcommand{\bM}{{\mathbf M}}
\newcommand{\bN}{{\mathbf N}}
\newcommand{\bP}{{\mathbf P}}

\newcommand{\bR}{{\mathbf R}}

\newcommand{\bT}{{\mathbf T}}

\newcommand{\bV}{{\mathbf V}}

\newcommand{\bX}{{\mathbf X}}
\newcommand{\bY}{{\mathbf Y}}
\newcommand{\bZ}{{\mathbf Z}}
\newcommand{\sE}{{\textrm{E}}}

\newcommand{\bone}{{\mathbf 1}}
\newcommand{\bzero}{{\mathbf 0}}

\newcommand{\bbeta}{{\boldsymbol \beta}}

\newcommand{\btheta}{{\boldsymbol \theta}}

\newtheorem{theorem}{Theorem}
\newtheorem{lemma}[theorem]{Lemma}
\newtheorem{corollary}[theorem]{Corollary}

\newcommand{\blind}{1}

\usepackage[bottom=1.2in]{geometry}

\begin{document}

\def\spacingset#1{\renewcommand{\baselinestretch}%
{#1}\small\normalsize} \spacingset{1}

%%%%%%%%%%%%%%%%%%%%%%%%%%%%%%%%%%%%%%%%%%%%%%%%%%%%%%%%%%%%%%%%%%%%%%%%%%%%%%

\if1\blind
{
  \title[A Diagnostic for Bias in Linear Mixed Model Estimators]{\bf A Diagnostic for Bias in Linear Mixed Model Estimators Induced by Dependence Between the Random Effects and the Corresponding Model Matrix}
\author{Andrew T. Karl}
\address{Adsurgo LLC}
\author{Dale L. Zimmerman}
\address{University of Iowa}
  \maketitle
} \fi

\if0\blind
{
  \bigskip
  \bigskip
  \bigskip
  \begin{center}
    {\LARGE\bf A Diagnostic for Bias in Linear Mixed Model Estimators Induced by Dependence Between the Random Effects and the Corresponding Model Matrix}
\end{center}
  \medskip
} \fi

\bigskip
\begin{abstract}
We explore how violations of the often-overlooked standard assumption that the random effects model matrix in a linear mixed model is fixed (and thus independent of the random effects vector) can lead to bias in estimators of estimable functions of the fixed effects. However, if the random effects of the original mixed model are instead also treated as fixed effects, or if the fixed and random effects model matrices are orthogonal with respect to the inverse of the error covariance matrix (with probability one), or if the random effects and the corresponding model matrix are independent, then these estimators are unbiased. The bias in the general case is quantified, and a diagnostic for its random-effects contribution is compared to a randomized permutation distribution of the predicted random effects, producing an informative summary graphic for each estimator of interest. This is demonstrated through the examination of sporting outcomes used to estimate a home field advantage.
\end{abstract}

\medskip
{\footnotesize
\noindent\textbf{Revision note (16 October 2026).}
This paper was published in the \emph{Journal of Statistical Planning and Inference}, 211 (2021), 107--118, doi:\url{10.1016/j.jspi.2020.06.004}.
This version corrects an omitted term in Theorem~2 and makes explicit in Theorem~6 the assumption that $\bds{\epsilon}$ is independent of the pair $(\bse,\bZ)$; related proofs and the interpretation of the diagnostic have been adjusted accordingly. Section~\ref{sec:simtest} adds a brief note on the size of the added term in conditional simulations. Code, inputs and results for these simulations are included in the accompanying update to the Mendeley Data supplement. Two typographical errors are also corrected: the factor $\sigma^2$ in Corollary~8 and the fixed-effects simulation mean in Section~\ref{sec:sportssim}.
The omitted term was discovered by the first author using Codex GPT6 Astra to audit his published papers. Codex also suggested its repair, which was approved by both authors.
We reported the issues to JSPI on 30 September 2026. On 15 October 2026, the journal office informed us that the editor considered the error minor and did not recommend publication of a corrigendum.
\par}
\medskip

\noindent%
\vfill

\newpage
\spacingset{1.5} % DON'T change the spacing!

\section{Introduction}
\label{sec:intro}
Standard linear mixed models are built conditional on the model matrices for the fixed and random effects, meaning that these matrices are assumed to be fixed and constructed without reference to the anticipated errors or random effects.   Unless the model matrices for the fixed and random effects are orthogonal with respect to the inverse of the error covariance matrix, dependence between the random effects and their corresponding model matrix can induce bias in the estimators of estimable functions of the fixed effects. This paper develops graphical and numeric diagnostics for the bias of estimators of estimable functions of fixed effects in mixed models that are constructed under the assumption that the model matrices are fixed when in fact the random effects model matrix is stochastic.

Consider a linear mixed model with fixed model matrices:
\begin{equation}\label{eq:mixedmodel}
\bY=\bX\bsB+\bZ\bse+\bds{\epsilon}.
\end{equation}
Here $\bY$ is a continuous response, and $\bse$ and $\bds{\epsilon}$ are independent with $\bse\sim N_m(\bds{0},\sigma^2\bG)$ and $\bds{\epsilon}\sim N_n(\bds{0},\sigma^2\bR)$. Although we assume normality for these vectors in order to match the most common applications, the main results on bias do not depend on this assumption.  The matrices $\bG$ and $\bR$ are assumed to be positive definite, $\sigma^2$ is assumed to be positive, and neither $\bX$ nor $\bZ$ are required to be full rank. If	 $\bk'\bsB$ is estimable under this model and $\bV\equiv(1/\sigma^2)\mbox{var$(\bY)$}=\bZ\bG\bZ'+\bR$ is known, then the best linear unbiased estimator (BLUE) of $\bk'\bsB$ is $\bk'\hat{\bsB}$ where $\hat{\bsB}=\left(\bX'\bV^{-1}\bX\right)^{-}\bX'\bV^{-1}\bY$. (Here and throughout, for any matrix $\bM$, $\bM^-$ represents an arbitrary generalized inverse of $\bM$.)

In many important applications of the linear mixed model, $\bR$ is known (often, in fact, $\bR=\bI$) but $\bG$ is unknown; more precisely, the elements of $\bG$ are known functions of an unknown parameter $\btheta$, i.e., $\bG=\bG(\btheta)$.  For inference on $\bk'\bsB$ to proceed in such settings it is necessary to first obtain an estimate $\hat{\btheta}$ which can be substituted for the unknown $\btheta$ to obtain an estimate $\hat{\bG}=\bG(\hat{\btheta})$ and a corresponding estimate $\hat{\bV}=\bZ\hat{\bG}\bZ'+\bR$ of $\bV$; after that, an empirical BLUE (E-BLUE) of $\bk'\bsB$ may be calculated as $\bk'\hat{\hat{\bbeta}}$ where $\hat{\hat{\bbeta}}=\left(\bX'\hat{\bV}^{-1}\bX\right)^{-}\bX'\hat{\bV}^{-1}\bY$.  Though the E-BLUE generally is not linear or best in any sense, it is unbiased when $\bZ$ is fixed provided that $\hat{\btheta}$ is an even and translation-invariant estimator \citep{kackar81}.  (An estimator $\hat{\btheta}=\hat{\btheta}(\bY)$ of $\btheta$ is even and translation invariant if $\hat{\btheta}(-\bY)=\hat{\btheta}(\bY)$ and $\hat{\btheta}(\bY+\bX\bc)=\hat{\btheta}(\bY)$ for all $\bY$ and all $\bc$.)

As an alternative to a mixed effects model, a fixed effects model could be fit: 
\begin{equation}\label{eq:fixedmodel}
\bY=\bX_*\bsB_*+\bds{\epsilon}. 
\end{equation}
Here $\bds{\epsilon}\sim N_n(\bds{0},\sigma^2\bR)$, $\bsB'_*=[\begin{array}{cc}\bsB'& \bse'\end{array}]$, $\bX_*=[\begin{array}{cc} \bX& \bZ\end{array}]$, and $\bse$ is fixed.  If $\bk_*'\bsB_*$ is estimable under this model and $\bR$ is known, then the BLUE of $\bk_*'\bsB_*$ is $\bk_*'\tilde{\bsB}_*$ where $\tilde{\bsB}_*=[\begin{array}{cc}\tilde{\bsB}' & \tilde{\bse}'\end{array}]'=\left(\bX_*'\bR^{-1}\bX_*\right)^{-}\bX_*'\bR^{-1}\bY$. In order to consider only estimable functions of effects that are treated as fixed in the mixed effects model (\ref{eq:mixedmodel}), we will restrict attention to $\bk_*$ that satisfy  $\bk'_*=[\begin{array}{cc}\bk'& \bds{0}'\end{array}]$, where the length of the zero vector is equal to $m$, meaning that the BLUE of $\bk_*'\bsB_*$ is $\bk'\tilde{\bsB}.$  The decision to treat effects as fixed or random has been widely explored previously \citep{robinson1991,stroup, allison2014fixed}.

%We will see that this decision affects the sensitivity of the $\bsB$ estimators that are obtained under the usual assumption that $\bX_{\bse}$ and $\bZ_{\bse}$ are independent of $\bse$, depending both on how the levels of $\bse$ are sampled  and on how $\bX_{\bse}$, $\bZ_{\bse}$, and $\bR$ are constructed. 

The E-BLUE $\bk'\hat{\hat{\bsB}}$ and BLUE $\bk'\tilde{\bsB}$ of, respectively, estimable functions $\bk'{\bsB}$ and $\bk'_*{\bsB}_*$ under the mixed and fixed effects models (\ref{eq:mixedmodel}) and (\ref{eq:fixedmodel}) have differing bias characteristics under more general versions of these models in which $\bZ$ is stochastic  \citep{ allison2,lock2007}.  Henceforth we refer to these more general versions as stochastic-$\bZ$ mixed and fixed effects models.  This issue seems to have received more attention in the economics literature \citep{wu,hausman,hausman2,allison2, wooldridge,lock2007} than in standard statistics textbooks on linear mixed models \citep{verb99,sasbook,stroup,demidenko}. Our aim is to increase this awareness by discussing an application with clear visual evidence and by proposing computationally-light diagnostics and graphics that statistical software could produce in order to help quickly detect bias in $\bk'\hat{\hat{\bsB}}$ on an application-by-application basis.

%This paper derives the fixed effect bias in mixed models and explores how additional output could be added to standard mixed model software in order to provide an estimate of this bias by using quantities already calculated by the software. 

%For the application, suppose $\bY$ is the margin of home-victory in a season of a sport, $\bse$ is a vector of team ratings, $\bsB=\beta$ is a home field advantage, and $\bZ_{\bse}$ will be next year's schedule. The maker of the schedule has a good sense of what $\bse$ will look like, based on how teams have performed this year and in the past. 

Section~\ref{sec:HFE} provides a practical motivation for considering this problem by comparing the fixed and mixed effects model estimates of home field scoring advantage for several sports. Section~\ref{sec:mme} derives sampling properties, including bias, of the E-BLUE and BLUE of estimable functions under the stochastic-$\bZ$ mixed and fixed effects models. Section~\ref{sec:biasdiagnostics} develops bias diagnostics based on a randomized permutation distribution of predicted random effects and on simulation. Section~\ref{sec:cont} applies these results to the home field advantage problem. Section~\ref{sec:sportssim} simulates the home field advantage problem in order to illustrate the findings of Section~\ref{sec:mme} by manipulating the team schedules, $\bZ$. An appendix describes simulations that investigate the power of the randomization test.

\section{Estimating Home Field Advantage}\label{sec:HFE}

There has long been observed a ``home-field advantage'' across a variety of sports \citep{lopez2018}.  We define home field scoring advantage (HFA) as the difference in the expected scores of the home and away teams within a game, after the strength of each team has been accounted for.  It is possible to account for the team strengths with either fixed or random effects. This section presents both a fixed effects model and a mixed effects model for the home team margins of victory, $Y_i=y_{H_i}-y_{A_i}$, where $y_{H_i}$ and $y_{A_i}$ are the home and away team scores, respectively, for game $i=1,...,n$. With $m$ teams in a data set, the pattern of game opponents and locations (the schedule) is recorded in an $n\times m$ matrix $\bZ$ as follows: if team $T_H$ hosted team $T_A$ in game $i$, then the $i$-th row of $\bZ$ consists of all zeros except for a $1$ in the column corresponding to $T_H$ and a $-1$ in the column corresponding to $T_A$.  To simplify the discussion, neutral site games are not considered.

\subsection{Fixed Effects Model}
The first plausible model we consider for HFA includes a fixed effect $\lambda$ for the HFA and a vector of fixed team strength effects $\bsB=\left(\beta_1,\ldots,\beta_m\right)'$, where the difference between any two team strength effects represents the expected difference in score in a game between the two teams on a neutral field. For each game $i$, this model assumes
\begin{equation*}
Y_i\sim N\left(\lambda + \beta_{H_i} - \beta_{A_i}, \sigma^2 \right).
\end{equation*}
Here $H_i$ and $A_i$ are the indices for the home and away teams, respectively, in game $i$. The errors are assumed to be independent, leading to an overall model 
\begin{equation}\label{eq:femodel}
\bY=\bone\lambda + \bZ\bds{\beta} + \bds{\epsilon}.
\end{equation}
Here $\bY=(Y_1,\ldots,Y_n)'$, $\bZ$ is fixed, and $\bds{\epsilon}\sim N(\bds{0},\sigma^2\bI)$. This model appears previously as Model 2 of \cite{harville94} and Model 1 of \cite{harville03}. The model matrix $\bX_*=[\bone,\bZ]$  is not full rank and the individual team strength effects $\beta_i$ are not estimable. However, $\lambda$ is estimable, as are the pairwise differences between team strength effects, provided there is sufficient mixing of the teams. For example, in a sport of three teams \{A, B, C\}, $\lambda$ will be estimable if each pair of teams have a home and home series, or if, for example, team B plays at team A, team A plays at team C, and team C plays at team B. \citet[Section 5.2.1]{stroup} provides a good summary of the estimability of fixed effects. 

\subsection{Mixed Effects Model}
Alternatively, the team strengths may be modeled as random effects $\bds{\eta}=\left(\eta_1,\ldots,\eta_m
\right)'$ that are assumed to be independent of the errors, $\bds{\epsilon}$.  Then, $Y_i$ is modeled conditional on the random effects as
\begin{equation*}
Y_i|\bse\sim N\left(\lambda + \eta_{H_i} - \eta_{A_i}, \sigma^2 \right).
\end{equation*}
This produces an overall model
\begin{align}\label{eq:mixed}
\bY|\bse&\sim N(\bone\lambda + \bZ\bds{\eta},\sigma^2\bI)\\
\bse&\sim N\left(\bzero,\sigma^2_g\bI\right).\nonumber
\end{align}
As in the fixed effects model, $\bZ$ is assumed to be fixed and $\lambda$ is estimable.
\citet[Equation 2.1]{harville77} considers a generalization of this model for ranking teams. An advantage of placing a distributional assumption on the team strength effects is that it provides a form of regularization for the model and avoids the aforementioned estimability concern.  A disadvantage is that another parameter ($\sigma_g^2$) is introduced that must be estimated before inference on $\lambda$ can proceed.  Here $\bV=\theta\bZ\bZ'+\bI$ where $\theta=\sigma_g^2/\sigma^2$, so there is a single variance component ratio, $\theta$, to estimate.  Commonly, $\theta$ is estimated by the methods of maximum likelihood or residual maximum likelihood (REML); both methods yield even, translation-invariant estimators \citep{kackar81}.

\subsection{Results}

\begin{figure}
\caption{: EBLUPs of the team ratings from the mixed model for the 2017 men's college basketball season plotted by proportion of games played at home by the team, along with a LOESS smoother.}
\label{fig:rating_by_proprotion}
\includegraphics[scale=.55]{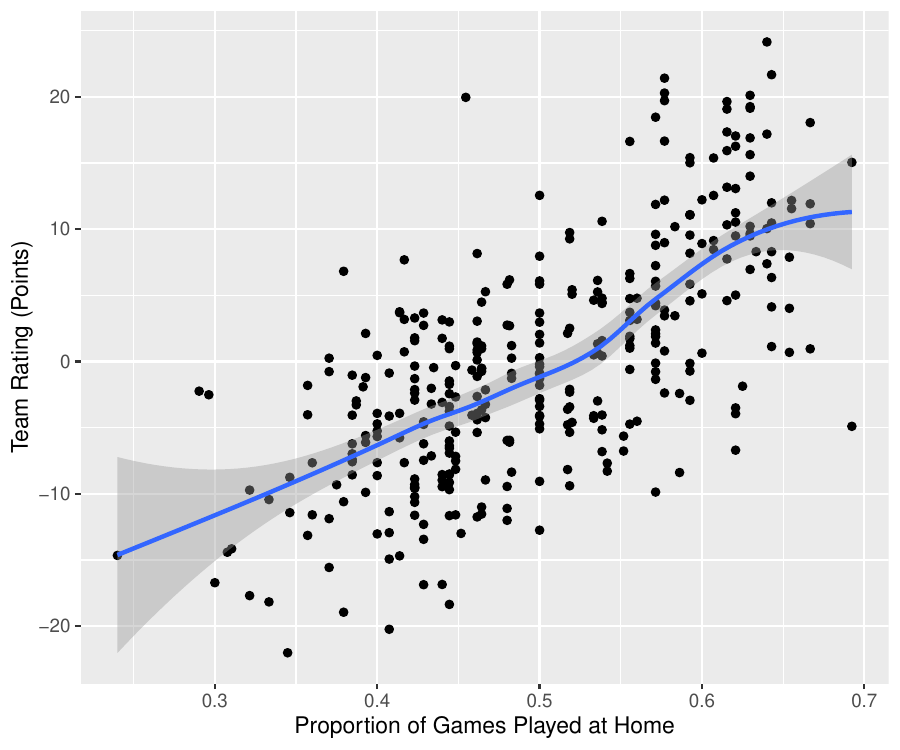}
\end{figure}

Scores from the 2000-2017 regular seasons (excluding playoffs) for each of the six sports of Men's and Women's NCAA Division I basketball, NCAA Football Bowl Subdivision (FBS), the National Football League (NFL), the Women's National Basketball Association (WNBA), and the National Basketball Association (NBA) are furnished by the website of \citet{masseyr}, with some missing WNBA results obtained from WNBA.com. In all cases, neutral site games are excluded in order to simplify the discussion of the application. The supplementary material \citep{kzdata} contains code for downloading the available scores from \cite{masseyr} and for obtaining the REML estimates of the models in Equations \ref{eq:femodel} and \ref{eq:mixed} with the R package \texttt{mvglmmRank} \citep{broatch2}.

Generally, professional teams play balanced schedules with an equal number of home and away games and a representative selection of teams from the league. In the regular season, the NBA games are perfectly balanced in the sense that every team plays 41 home and 41 away games (with occasional exceptions for neutral site games). By contrast, college sports tend to have unbalanced schedules, with the better and more influential teams able to schedule more home than away games.  For example, 10\% of 2017 Division I men's college basketball teams played no more than 39\% of their games at home, while another 10\% of teams played no fewer than 62\% of their games at home. Furthermore, there is a tendency for influential college teams to intentionally schedule weaker opponents for their surplus home games. Athletic directors of better teams accomplish this by building their schedules ($\bZ$) for future years using historical knowledge of the quality of their potential opponents and the general tendency for team quality to be (positively) correlated over years.  This results in a violation of the assumption in the mixed model (\ref{eq:mixed}) that $\bZ$ is fixed, or at least independent of $\bse$.   To illustrate, Figure~\ref{fig:rating_by_proprotion} plots the empirical best linear unbiased predictors (EBLUPs) of the 2017 men's college basketball team ratings ($\hat{\bse}$) from the mixed model against the proportion of games played at home (a function of $\bZ$).

\begin{figure}
\caption{: BLUEs of the HFA under the fixed effects model (\ref{eq:femodel}) (solid blue line) are plotted with the corresponding E-BLUEs under the mixed effects model (\ref{eq:mixed}) (dashed black line) across seasons of college sports (left) and professional sports (right). The lines in the professional plots nearly coincide. }
\label{plot:marg1}
\centering
\includegraphics[scale=.93]{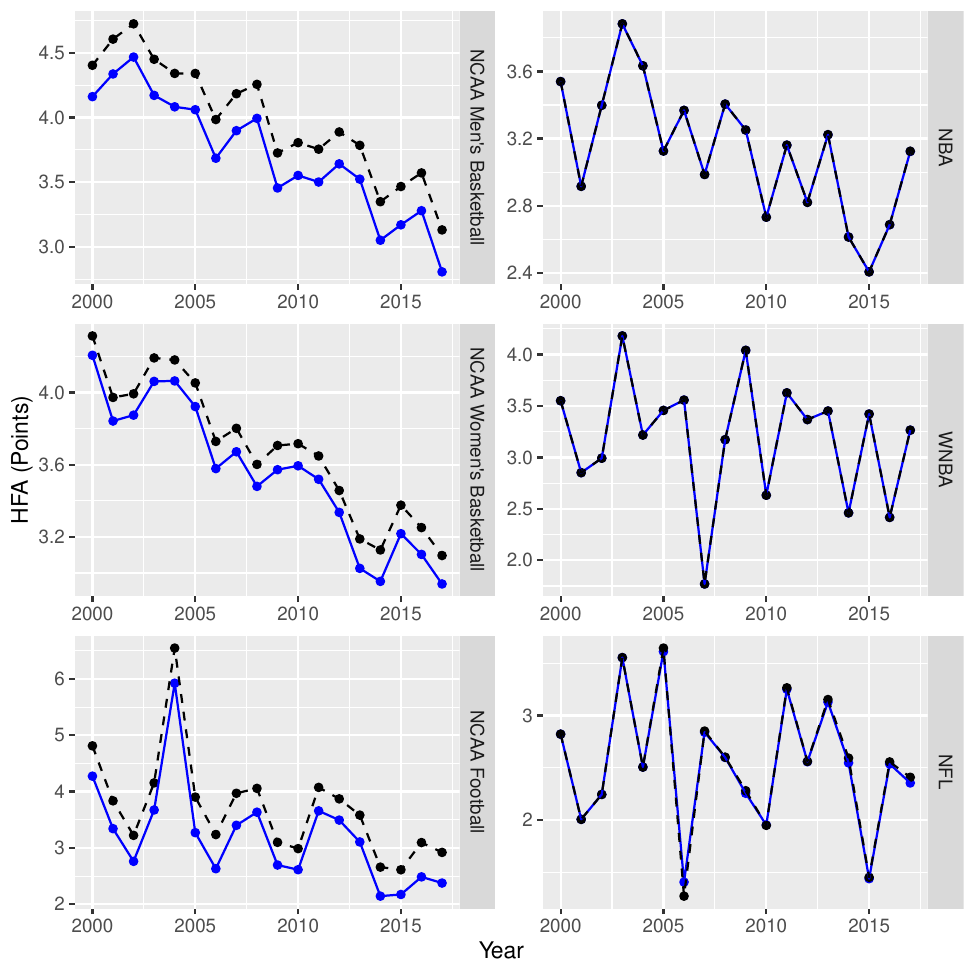}

\end{figure}

Figure~\ref{plot:marg1} shows the HFA estimates from the fixed and mixed effects models for each season in each sport, including downward trends in the college HFAs that have been noted by sports journalists \citep{omaha,si}. Estimates of $\lambda$ from the two models are mostly identical in the professional sports, but are cleanly separated in the college sports where the mixed-model estimates (the E-BLUEs) are uniformly larger than the fixed effect estimates (the BLUEs).  Interestingly, the separation between the fixed and mixed model estimates of HFA is relatively constant across years within the sports.  Section~\ref{sec:mme} demonstrates that when $\bZ$ is stochastic, the E-BLUE obtained by fitting the mixed effects model is generally biased but the BLUE obtained by fitting the fixed effects model is unbiased.

\section{Sampling Properties of $\bk'\hat{\hat{\bsB}}$ and $\bk'\tilde{\bsB}$ under Stochastic-$\bZ$ Models}\label{sec:mme}
Throughout this section, let $\bk'\bsB$ be a linear function of $\bsB$ that is estimable under the standard (fixed-$\bZ$) mixed effects model (\ref{eq:mixedmodel}), and suppose that $\bk'\bsB$ is also estimable with probability one under the stochastic-$\bZ$ version of the fixed effects model (\ref{eq:fixedmodel}).  (The HFA described in the previous section is such a function.)  Furthermore, let $\bk'\hat{\hat{\bsB}}$ and $\bk'\tilde{\bsB}$ be the E-BLUE and BLUE of $\bk'\bsB$, respectively, obtained by fitting standard versions of those models, where the estimator $\hat{\btheta}$ used to obtain $\bk'\hat{\hat{\bsB}}$ is even and translation-invariant.  In this section, we investigate sampling properties, specifically the bias and the variance, of the E-BLUE and BLUE under stochastic-$\bZ$ versions of the models.

\subsection{Bias}

\begin{lemma}\label{lemma:ginv}
Let $\bM$ be a generalized inverse of $\bA'\bA$ where $\bA$ is any matrix. Then $\bM\bA'$ is a generalized inverse of $\bA$, so that $\bA\bM\bA'\bA=\bA$.
\end{lemma}
\noindent The proof of Lemma~\ref{lemma:ginv} can be found in many linear algebra texts, including \citet{har}.

\begin{theorem}\label{theorem2}
Let $\hat{\bsn}'=\bk'(\bX'\hat{\bV}^{-1}\bX)^{-}\bX'\hat{\bV}^{-1}\bZ$.  If $\bZ$ is stochastic in the mixed effects model but $\bds{\epsilon}$ and $\bZ$ are independent, then the bias of $\bk'\hat{\hat{\bsB}}$ is equal to \normalfont $\sE\left[\hat{\bsn}'\bse\right]+C_k$, where
\[
C_k=\sE\left[\bk'(\bX'\hat{\bV}^{-1}\bX)^{-}\bX'\hat{\bV}^{-1}\bds{\epsilon}\right].
\]
\end{theorem}
\begin{proof}
Since $\bk'\bsB$ is estimable under model (\ref{eq:mixedmodel}), there exists a vector $\bt$ such that $\bk'=\bt'\bX$. And since $\hat{\bV}$ is positive definite (with probability one) with Cholesky decomposition  $\hat{\bV}=\hat{\bL}\hat{\bL}'$ (where $\hat{\bL}$ is lower triangular with positive diagonal entries), there also exists (with probability one) a vector $\bt_{\hat{\bL}}=\hat{\bL}'\bt$ such that $\bk'=\bt_{\hat{\bL}}'\hat{\bL}^{-1}\bX$. Then
\begin{align}
\sE[\bk'\hat{\hat{\bsB}}]&=\sE[\sE[\bk'(\bX'\hat{\bV}^{-1}\bX)^{-}\bX'\hat{\bV}^{-1}\bY|\bse,\bZ]]\nonumber\\
&=\sE[\sE[\bk'(\bX'\hat{\bV}^{-1}\bX)^{-}\bX'\hat{\bV}^{-1}(\bX\bsB+\bZ\bse+\bds{\epsilon})|\bse,\bZ]]\nonumber\\
&=\sE[\bk'(\bX'\hat{\bV}^{-1}\bX)^{-}\bX'\hat{\bV}^{-1}(\bX\bsB+\bZ\bse)]+C_k\label{eq:bias0}\\
&=\sE[\bt_{\hat{\bL}}'\hat{\bL}^{-1}\bX(\bX'\hat{\bL}'^{-1}\hat{\bL}^{-1}\bX)^{-}\bX'\hat{\bL}'^{-1}\hat{\bL}^{-1}(\bX\bsB+\bZ\bse)]+C_k\nonumber\\
&=\bk'\bsB+\sE[\hat{\bsn}'\bse]+C_k.\label{eq:bias}
\end{align}
Equation~\ref{eq:bias0} separates the residual contribution $C_k$. Equation~\ref{eq:bias} uses Lemma~\ref{lemma:ginv} with $\bA=\hat{\bL}^{-1}\bX$. 
\end{proof}

If $\btheta$ is known, $\hat{\bV}=\bV$ depends only on $\bZ$, so independence of $\bds{\epsilon}$ and $\bZ$ gives
\[
C_k=\sE\left[\bk'(\bX'\bV^{-1}\bX)^{-}\bX'\bV^{-1}\sE(\bds{\epsilon}\mid\bZ)\right]=0.
\]
When $\btheta$ is estimated from the response, $C_k$ need not be zero, even if $\bds{\epsilon}$ is independent of the pair $(\bse,\bZ)$.

Our first corollary to Theorem 2 reveals that bias in $\bk'\hat{\hat{\bsB}}$ does not occur merely because $\bZ$ is stochastic.

\begin{corollary}\label{t:2}
If $\bZ$ is stochastic in the mixed effects model but $\bds{\epsilon}$, $\bse$, and $\bZ$ are independent, then $\bk'\hat{\hat{\bsB}}$ is unbiased.
\end{corollary}
\begin{proof}
\begin{align}
\sE[\hat{\bsn}'\bse]=&\sE[\bk'(\bX'\hat{\bV}^{-1}\bX)^{-}\bX'\hat{\bV}^{-1}\bZ\bse]\nonumber\\
=&\sE\{\sE[\bk'(\bX'\hat{\bV}^{-1}\bX)^{-}\bX'\hat{\bV}^{-1}\bZ\bse|\bZ]\}\nonumber\\
=&\sE\{0\}\label{eq:bias1}\\
=&0.\nonumber
\end{align}
Equation~\ref{eq:bias1} uses the symmetry argument of \citet{kackar81} conditional on $\bZ$. The joint sign change $(\bse,\bds{\epsilon})\mapsto(-\bse,-\bds{\epsilon})$ preserves the conditional distribution of $(\bse,\bds{\epsilon})$ given $\bZ$ and, by evenness and translation invariance of $\hat{\btheta}$, leaves $\hat{\bV}$ unchanged. Both $\hat{\bsn}'\bse$ and $\bk'(\bX'\hat{\bV}^{-1}\bX)^{-}\bX'\hat{\bV}^{-1}\bds{\epsilon}$ change sign. Thus $C_k=0$ as well, and the result follows by Theorem 2.
\end{proof}

 %The independence of $\bse$ and $\bZ$ is often assumed without verification in applications of mixed models with standard software, suggesting that practitioners may benefit from the availability of diagnostic information about $\hat{\bsn}'\hat{\bse}$.

Our second corollary to Theorem 2 shows that $\bk'\hat{\hat{\bsB}}$ can be unbiased even when $\bZ$ and $\bse$ are dependent, provided that a certain orthogonality condition holds.

\begin{corollary}\label{t:3}
If $\bZ$ is stochastic in the mixed effects model but $\bds{\epsilon}$ and $\bZ$ are independent and $\bZ$ is orthogonal to $\bX$ with respect to $\bR^{-1}$ (with probability one), then $\bk'\hat{\hat{\bsB}}$ is unbiased.
\end{corollary}
\begin{proof}
By Theorem 18.2.8 of \citet{har}, $\bX'\hat{\bV}^{-1}\bZ=\bX'\bR^{-1}\bZ -\bX'\bR^{-1}\bZ\hat{\bT}\bZ'\bR^{-1}\bZ$ where $\hat{\bT}=(\bZ'\bR^{-1}\bZ+\hat{\bG}^{-1})^{-1}$. Since $\bX'\bR^{-1}\bZ=\bds{0}$ except on a set of probability zero, $\bX'\hat{\bV}^{-1}\bZ=\bds{0}-\bds{0}\hat{\bT}\bZ'\bR^{-1}\bZ=\bds{0}$, which implies that $\hat{\bsn}=\bds{0}$ with probability 1. Thus, $\sE\left[\hat{\bsn}'\bse\right]=0$. Moreover, $\bX'\bR^{-1}\hat{\bV}=\bX'\bR^{-1}(\bZ\hat{\bG}\bZ'+\bR)=\bX'$, so $\bX'\hat{\bV}^{-1}=\bX'\bR^{-1}$ and
\[
C_k=\sE\left[\bk'(\bX'\bR^{-1}\bX)^{-}\bX'\bR^{-1}\bds{\epsilon}\right]=0.
\]
The result follows by Theorem 2.
\end{proof}

By contrast with the E-BLUE obtained by fitting a mixed effects model with stochastic $\bZ$, the next theorem reveals that the BLUE $\bk'\tilde{\bsB}$ obtained by fitting the fixed effects model is unbiased when $\bZ$ is stochastic without any conditions on $\bse$ (other than it being fixed).

\begin{theorem}\label{t:fixed}
If $\bZ$ is stochastic in the fixed effects model but $\bds{\epsilon}$ and $\bZ$ are independent, then $\bk'\tilde{\bsB}$ is unbiased.
\end{theorem}
\begin{proof}
Since, by assumption, $\bk'\bsB$ is estimable with probability one under the fixed effects model with stochastic $\bZ$, for every $\bZ$ on a set of probability one there exists a vector $\bt_{\bZ}$ such that
$\bk'=\bt_{\bZ}'(\bI-\bP_{\bZ_@})\bX_@$, where $\bP_{\bZ_@}=\bZ_@(\bZ_@'\bZ_@)^{-}\bZ_@'$, $\bZ_@=\bR^{-1/2}\bZ$, and $\bX_@=\bR^{-1/2}\bX$.  Thus, upon defining $\bds{\epsilon}_@=\bR^{-1/2}\bds{\epsilon}$, we find that
\begin{align}
\sE(\bk'\tilde{\bsB})&=\sE\{\sE\{\bk'[\bX_@'(\bI-\bP_{\bZ_@})\bX_@]^{-}\bX_@'(\bI-\bP_{\bZ_@})(\bX_@\bsB+\bZ_@\bse+\bds{\epsilon}_@)|\bZ\}\}\nonumber\\
&=\bk'\bsB+\sE\{\sE[\bk'[\bX_@'(\bI-\bP_{\bZ_@})\bX_@]^{-}\bX_@'(\bI-\bP_{\bZ_@})\bds{\epsilon}_@|\bZ]\}\nonumber\\
&=\bk'\bsB+\sE\{\bk'[\bX_@'(\bI-\bP_{\bZ_@})\bX_@]^{-}\bX_@'(\bI-\bP_{\bZ_@})\}\sE(\bds{\epsilon}_@)\nonumber\\
&=\bk'\bsB.\nonumber
\end{align}
\end{proof}

\subsection{Variance}\label{sec:variance}
The variances of $\bk'\hat{\hat{\bsB}}$ and $\bk'\tilde{\bsB}$ are also affected by the stochasticity of $\bZ$ in the mixed and fixed effects models.  However, the variance of $\bk'\hat{\hat{\bsB}}$ is also affected by the estimation of $\btheta$; in fact, under the standard fixed-$\bZ$ model, if $\hat{\btheta}$ is even and translation-invariant, then var$(\bk'\hat{\hat{\bsB}})$ is no smaller than if $\btheta$ were known \citep{kackar84}.  In order to focus exclusively on the effect of a stochastic $\bZ$ on this variance, in this section we assume that $\btheta$ is known.  Results comparing the variances of the E-BLUEs of $\bk'\bsB$ under the fixed-$\bZ$ and stochastic-$\bZ$ models when $\btheta$ is known should shed some light on how those variances compare when $\btheta$ is unknown. 
\begin{theorem}\label{t:4}
If $\bZ$ is stochastic in the mixed effects model but the conditions of Theorem 2 hold, $\bds{\epsilon}$ is independent of the pair $(\bse,\bZ)$, and $\btheta$ is known, the variance of the BLUE, $\bk'\hat{\bsB}$, is \normalfont $\sigma^2\textrm{E}[\bk'(\bX'\bV^{-1}\bX)^-\bk] + \Var(\bsn'\bse)-\sigma^2\textrm{E}(\bsn'\bG\bsn)$, where ${\bsn}'=\bk'(\bX'{\bV}^{-1}\bX)^{-}\bX'{\bV}^{-1}\bZ$. 
\end{theorem}

\begin{proof}
The additional independence assumption gives $\sE(\bds{\epsilon}\mid\bse,\bZ)=\bds{0}$ and $\Var(\bds{\epsilon}\mid\bse,\bZ)=\sigma^2\bR$. Hence
\begin{align*}
\Var(\bk'\hat{\bsB})=&\Var\{\textrm{E}[\bk'(\bX'\bV^{-1}\bX)^-\bX'\bV^{-1}\bY|\bse,\bZ]\}+\textrm{E}\{\Var[\bk'(\bX'\bV^{-1}\bX)^-\bX'\bV^{-1}\bY|\bse,\bZ]\}\\
=&\Var[\bk'(\bX'\bV^{-1}\bX)^-\bX'\bV^{-1}(\bX\bsB+\bZ\bse)]+\textrm{E}\{\Var[\bk'(\bX'\bV^{-1}\bX)^-\bX'\bV^{-1}\bds{\epsilon}|\bse,\bZ]\}\\
=&\Var(\bk'\bsB+\bsn'\bse)+\textrm{E}\{\bk'(\bX'\bV^{-1}\bX)^-\bX'\bV^{-1}\sigma^2[\bV-\bZ\bG\bZ']\bV^{-1}\bX[(\bX'\bV^{-1}\bX)^-]'\bk\}\\
=&\sigma^2\textrm{E}[\bk'(\bX'\bV^{-1}\bX)^-\bk]+\Var(\bsn'\bse)-\sigma^2\textrm{E}(\bsn'\bG\bsn).
\end{align*}
\end{proof}

%The bias takes the form of a generalized least square regression of $\bZ\bse$ against $\bX$ with error covariance matrix $\bV$. Under the hypothesis that the bias is $\bds{0}$, then $\bZ\bse\sim N(\bds{0},\bV)$, and $\bds{s}=\bL^{-1}\bZ\bse\sim N(\bds{0},\bI)$. We should be able to fit $\bds{s}$ against any factor of interest without seeing a significant result.

\begin{corollary}\label{t:5}
If $\bZ$ is stochastic in the mixed effects model but $\bds{\epsilon}$, $\bse$, and $\bZ$ are independent, then \normalfont $\Var(\bk'\hat{\bsB})=\sigma^2\textrm{E}[\bk'(\bX'\bV^{-1}\bX)^-\bk]$.
\end{corollary}
\begin{proof}
\begin{align*}
\Var\left(\bsn'\bse\right)=&\sE\left[\bsn'\bse\bse'\bsn\right]-\left(\sE\left[\bsn'\bse\right]\right)^2\\
=&\sE\left[\Tr\left(\bsn\bsn'\bse\bse'\right)\right]\\
=&\tr\left(\bE\left[\bsn\bsn'\right]\sigma^2\bG\right)\\
=&\sigma^2\sE\left[\bsn'\bG\bsn\right].
\end{align*}
This implies that
$\Var(\bk'\hat{\bsB})=\sigma^2\textrm{E}[\bk'(\bX'\bV^{-1}\bX)^-\bk]$ by Theorem~\ref{t:4}.
\end{proof}

%\noindent Additionally, the Cauchy-Schwarz inequality puts an upper bound on the magnitude of the expected bias, $|\bE[\bsn'\bse]|\leq \sqrt{\bE[\bsn'\bsn]\bE[\bse'\bse]}$, with equality when $\bsn$ and $\bse$ are linearly dependent or when $\bsn=\bds{0}$.

\begin{corollary}\label{t:6}
If $\bZ$ is stochastic in the mixed effects model, but $\bds{\epsilon}$ and $\bZ$ are independent and $\bZ$ is orthogonal to $\bX$ with respect to $\bR^{-1}$ (with probability one), then \normalfont $\Var(\bk'\hat{\bsB})=\sigma^2\textrm{ E}[\bk'(\bX'\bV^{-1}\bX)^-\bk]-\sigma^2\textrm{E}(\bsn'\bG\bsn)$.
\end{corollary}
\begin{proof}
As in the proof of Corollary 4, orthogonality gives $\bsn=\bds{0}$ and $\bX'\bV^{-1}=\bX'\bR^{-1}$. Thus
\[
\bk'\hat{\bsB}=\bk'\bsB+\bk'(\bX'\bR^{-1}\bX)^{-}\bX'\bR^{-1}\bds{\epsilon}.
\]
Its variance is $\sigma^2\bk'(\bX'\bR^{-1}\bX)^{-}\bk$. Since $\bX'\bV^{-1}\bX=\bX'\bR^{-1}\bX$ and $\bsn=\bds{0}$, this is equivalently
\begin{align*}
\Var(\bk'\hat{\bsB})=& \sigma^2\textrm{ E}[\bk'(\bX'\bV^{-1}\bX)^-\bk]-\sigma^2\textrm{E}(\bsn'\bG\bsn).
\end{align*}
\end{proof}

\noindent Under the conditions of Corollary~\ref{t:6}, the variance of $\bk'\hat{\bsB}$ reduces to $\sigma^2\bk'(\bX'\bR^{-1}\bX)^{-}\bk$. For fixed $\bX$, $\bR$, and $\sigma^2$, this is no larger than the variance under the conditions of Corollary~\ref{t:5}, since $\bV-\bR=\bZ\bG\bZ'$ is positive semidefinite.

\begin{theorem}\label{t:7}
If $\bZ$ is stochastic in the fixed effects model but $\bds{\epsilon}$ and $\bZ$ are independent, then \normalfont $\Var(\bk'\tilde{\bsB})=\sigma^2\sE\{\bk'[\bX'\bR^{-1}\bX-\bX'\bR^{-1}\bZ(\bZ'\bR^{-1}\bZ)^-\bZ'\bR^{-1}\bX]^{-} \bk\}$.
\end{theorem}
\begin{proof}
Using the same notation as in the proof of Theorem~\ref{t:fixed},
\begin{align*}
\Var(\bk'\tilde{\bsB})=&\sE\{\sE\{\bk'[\bX_@'(\bI-\bP_{\bZ_@})\bX_@]^{-}\bX_@'(\bI-\bP_{\bZ_@})\bds{\epsilon}_@\bds{\epsilon_@}'\\
&\times (\bI-\bP_{\bZ_@})\bX_@[\bX_@'(\bI-\bP_{\bZ_@})\bX_@]^{-}\bk|\bZ\}\}+0\\
=&\sigma^2\sE\{\bk'[\bX_@'(\bI-\bP_{\bZ_@})\bX_@]^{-}\bX_@'(\bI-\bP_{\bZ_@})\bX_@[\bX_@'(\bI-\bP_{\bZ_@})\bX_@]^{-}\bk\}\\
=&\sigma^2\sE\{\bk'[\bX_@'(\bI-\bP_{\bZ_@})\bX_@]^{-}\bk\}\\
=&\sigma^2\sE\{\bk'\{\bX'\bR^{-1/2}[\bI-\bR^{-1/2}\bZ(\bZ'\bR^{-1}\bZ)^{-}\bZ'\bR^{-1/2}]\bR^{-1/2}\bX\}^{-}\bk\}\\
=&\sigma^2\sE\{\bk'[\bX'\bR^{-1}\bX-\bX'\bR^{-1}\bZ(\bZ'\bR^{-1}\bZ)^-\bZ'\bR^{-1}\bX]^{-} \bk\}.
\end{align*}

\end{proof}
\noindent By Theorem~\ref{t:7}, the variance of $\bk'\tilde{\bsB}$ is equal to $\sigma^2\bk'(\bX'\bR^{-1}\bX)^-\bk$ when $\bX'\bR^{-1}\bZ=0$ (with probability 1). In the disallowed case where $\bZ=\bX$ (with probability 1), the expression inside of the generalized inverse is $\bds{0}$.

\section{Bias Diagnostics}\label{sec:biasdiagnostics}
In situations where the mixed effects model E-BLUE ($\bk'\hat{\hat{\bsB}}$) is biased under the stochastic-$\bZ$ assumption but the fixed effects model BLUE ($\bk'\tilde{\bsB}$) is unbiased, the magnitude of the difference $\bk'\hat{\hat{\bsB}}-\bk'\tilde{\bsB}$ provides an estimate of the bias. \citet{hausman} uses this difference to formulate a specification test for the consistency of the generalized least squares estimators in a random intercept model.

Alternatively, the EBLUP of $\bse$ from the fitted mixed effects model can be substituted into $\sE[\hat{\bsn}'\bse]$ to produce a diagnostic value $\hat{\bsn}'\hat{\bse}$ that can serve as an approximate internal estimate of the bias of $\bk'\hat{\hat{\bsB}}$; it estimates the random-effects contribution $\sE[\hat{\bsn}'\bse]$ in Theorem~\ref{theorem2} and omits $C_k$. And although it is not pursued further here, the plug-in estimate $\hat{\bsn}$ of $\bsn$ might benefit from using the bias-corrected precision estimator of \citet{kr2} in place of $(\bX'\hat{\bV}^{-1}\bX)^{-}$.

\subsection{Randomization Test for Independence of $\bsn$ and $\bse$}\label{sec:randomization}

Even when $\bX$ and $\bZ$ are not orthogonal with respect to $\bR^{-1}$ (Corollary \ref{t:3}), $\bk'\hat{\hat{\bsB}}$ is still unbiased if $\bds{\epsilon}$, $\bse$, and $\bZ$ are mutually independent (Corollary \ref{t:2}). Under the null hypothesis of independence of $\bse$ and $\bsn$, the observed value $\hat{\bsn}'\hat{\bse}$ 
could be compared to the randomized permutation distribution \citep{edington} of 
\begin{equation*}
\hat{\bsn}'\bds{\pi}\left(\hat{\bse}\right).
\end{equation*}
Here $\bds{\pi}()$ is a permutation function and $\hat{\bsn}$ is held fixed at the value obtained during the original model fit. For a diagonal $\bG$, the permutations performed by $\bds{\pi}(\hat{\bse})$ are stratified within each random factor present in $\bse$ (corresponding to unique diagonal entries of $\bsG$). More generally, when $\bG$ has nonnull off-diagonal entries, the permutation function $\bds{\pi}(\hat{\bse})$ can be constructed as follows.
\begin{enumerate}
\item{Simulate a vector, $\bds{w_0}$, from $N(\bds{0},\hat{\bG})$.}
\item{Let $\bds{w}=\bds{w_0}$.}
\item{For each random factor (corresponding to the unique diagonal elements of $\bG$), replace the smallest element of $\bds{w}$ with the smallest element of the corresponding factor from $\hat{\bse}$.}
\item{Repeat Step 3  for the component of $\bds{w}$ corresponding to the second smallest entry in the original $\bds{w_0}$ for each factor, the third smallest, etc.}
\item{Return $\bds{w}$.}
\end{enumerate}
This function $\bds{\pi}(\hat{\bse})$ shuffles $\hat{\bse}$ within each factor according to the correlation structure assumed by $\bG$, provided that $\bG$ has equicorrelation within factors (meaning that random effects within a factor are exchangeable).

The permutation distribution provides an estimate of the distribution of the observed  $\hat{\bsn}'\hat{\bse}$ that would be expected under random sampling of $\bse$ (prior to the generation of $\bY$), and the mean of this distribution provides an estimate of $\sE[\hat{\bsn}'\bse]$ under random sampling of the random effects, a value that should equal 0 if the other conditions of Corollary~\ref{t:2} hold. The percentile of the observed value within the distribution of permuted values can be used to detect inconsistency with the hypothesis of independence of $\bse$ and $\bsn$. Small percentiles (say, $<0.5$), corresponding to unusually small values of the diagnostic, or large percentiles (say, $>99.5$), corresponding to unusually large values, indicate that $\bse$ and $\bsn$ may not be independent. As demonstrated in the next section, a plot of the randomized permutation distribution along with a vertical line at $\hat{\bsn}'\hat{\bse}$ provides a useful graphical summary for each estimable effect $\bk'\bsB$ of interest.

While the predictions $\hat{\bse}$ from many common mixed models sum to 0 within each random factor \citep{s}, this may not be true in all cases. If this sum is different from zero, then the randomization distribution of $\hat{\bsn}'\bds{\pi}\left(\hat{\bse}\right)$ may have mean different from zero, indicating that this diagnostic is not centered at zero under the proposed randomization. This may reflect some other departure from modeling assumptions, perhaps due to dependence between the model matrices and $\bds{\epsilon}$, or  due to the presence of a nonignorable missingness process.

\subsection{Simulation Based Estimate of Bias}\label{sec:simtest}
While Equation~\ref{eq:bias} suggests that $\hat{\bsn}'\hat{\bse}$ can serve as an internal estimate of the  bias for the associated estimator of the estimable function using only a single model fit, it is also possible to estimate the bias via simulation by using the model matrices and estimates from the original fit and repeatedly replacing the residual vector with a random vector drawn from $\bN(\bds{0},\hat{\sigma}^2\bR)$. Specifically,

\begin{enumerate}
\item{\label{item:21}Fit $\bsY|\bse \sim N\left(\bsX\bsB+\bsZ\bse,\sigma^2\bsR\right)$ where $\bse \sim N\left(\bds{0},{\sigma}^2\bsG\right)$ in order to obtain solutions $\hat{\hat{\bsB}}$, $\hat{\bse}$, and $\hat{\sigma}^2\bR$. The estimated random effects covariance matrix, $\hat{\sigma}^2\hat{\bG}$, is not used to generate the simulated responses.}
\item{\label{item:22} Repeatedly simulate a new response vector $\bsY_s=\bsX\hat{\hat{\bsB}}+\bsZ\hat{\bse} + \be_s$, where $\be_s$ is a random deviate from $N(\bds{0},\hat{\sigma}^2\bsR)$. Then fit $\bY_s$ using the same mixed model as Step 1 and record the estimate $\bk'\hat{\hat{\bsB_s}}$.}
\item{\label{item:23}Compare the mean of the sampling distribution of $\bk'\hat{\hat{\bsB_s}}$  to the target of $\bk'\hat{\hat{\bsB}}$ (the original estimate from $\bY$ in Step 1). The difference between $mean_s(\bk'\hat{\hat{\bsB}}_s)$  and $\bk'\hat{\hat{\bsB}}$ provides an estimate of the bias in the E-BLUE of $\bk'\bsB$ for this particular data set.} 
\end{enumerate}
The simulation includes the contribution corresponding to $C_k$ through covariance refitting, whereas the internal diagnostic omits it. It therefore estimates the full bias under its conditional plug-in generating model, with $\bZ$ and $\hat{\bse}$ held fixed.
Additional simulations using this conditional scheme for the five unbalanced 2017 sports schedules in Section~\ref{sec:cont} used $B=100{,}000$ replications each, with $\tilde{\bsB}$ as the generating fixed effects. The schedule, original fitted team effects, and generating residual covariance were held fixed; independent Gaussian errors were drawn and the covariance parameters re-estimated for each response. Let $\hat{\bV}_s$ denote $\hat{\bV}$ from replicate $s$. The row $\ba_0'$ is obtained from the formula for $\ba_s'$ below by replacing $\hat{\bV}_s$ with the original estimate $\hat{\bV}$, held fixed across replicates. We estimated the conditional counterpart of $C_k$ by
\begin{align*}
\ba_s'&=\bk'(\bX'\hat{\bV}_s^{-1}\bX)^{-}\bX'\hat{\bV}_s^{-1},\\
\widehat C_{k,\mathrm{cond}}&=\frac{1}{B}\sum_{s=1}^{B}(\ba_s-\ba_0)'\be_s.
\end{align*}
The subtracted term $\ba_0'\be_s$ uses the same simulated errors and has conditional mean zero because $\ba_0$ is fixed. It was used as a control variate to reduce Monte Carlo error. Code, inputs and saved results are provided in the accompanying supplement update.
The magnitude of the estimated residual-weight contribution was about $0.05\%$--$2.32\%$ of the magnitude of the Monte Carlo mean of the random-effects term $\hat{\bsn}'\bse$, with $\bse$ fixed at the original fitted team effects. For the balanced WNBA schedule, the residual-weight contribution is exactly zero. These conditional results do not establish that $C_k$ is generally negligible.
The advantage of the internal estimate of bias, $\hat{\bsn}'\hat{\bse}$, over the estimate obtained by simulation is that the model does not need to be fit repeatedly and it can be calculated using matrices that have already been produced by the original model fit. In steps 2 and 3, $\tilde{\bsB}$ could be used in place of $\hat{\hat{\bsB}}$ if it is available.

\section{Bias in the Mixed Model HFA Estimator}\label{sec:cont}
Table~\ref{tab:2017sport} contains the internal diagnostic $\hat{\bsn}'\hat{\bse}$ for each sport in 2017 from the application of Section~\ref{sec:HFE} (where $k=1$ and the estimable effect of interest is simply the intercept). These represent an internal estimate of the bias in the HFA E-BLUE. They are strongly correlated ($\rho=0.990$) with the difference between the mixed and fixed effects model estimates of HFA, and  with the simulation estimates ($\rho=0.996$), which represent external estimates of the mixed model bias. The percentile of the observed value $\hat{\bsn}'\hat{\bse}$  in the distribution of a sample of one million permutations $\hat{\bsn}'\bds{\pi}(\hat{\bse})$ is also reported, and is displayed for the 2017 seasons of the sports in Figure~\ref{fig:permsim}. The college sports produce observed values that are larger than all of the other sampled permutations of $\hat{\bse}$, providing strong evidence that  $\bsn$ and $\bse$ are not independent in these applications. Combined with the fact that none of the means of the sampling distributions differ substantially from zero (see the supplementary data tables and code provided by \citet{kzdata}, which also contains the results from the other 17 seasons of each sport), these results suggest that the mixed effects model E-BLUEs of HFA in Table~\ref{tab:2017sport} pertaining to the college sports are biased upward. Additionally, the sum of the team ratings $\hat{\bse}$ in each year is equal to 0, as expected \citep{s}.

% \usepackage{booktabs}
\begin{table}[]
\caption{: 2017 HFA estimates (in points) from the mixed and fixed effect models, along with the standard error of the mixed model estimates (calculated under the fixed-$\bZ$ assumption), the difference between the mixed and fixed HFA estimates, the internal bias diagnostics, and the percentile of the observed diagnostic in the randomized permutation distribution. The final two columns present the simulation results described in Section~\ref{sec:simtest}, with $\tilde{\bsB}$ used instead of $\hat{\hat{\bsB}}$ in Steps 2 and 3.}
\label{tab:2017sport}
\resizebox{\columnwidth}{!}{%
\begin{tabular}{@{}lllllllll@{}}
\toprule
\textbf{Sport} & $\begin{array}{l} \textrm{\textbf{Fixed}}\\(\tilde{\lambda})
\end{array}$ & $\begin{array}{l} \textrm{\textbf{Mixed}}\\(\hat{\hat{\lambda}})
\end{array}$ & $\begin{array}{l} \textrm{\textbf{Mixed}}\\\textrm{\textbf{S.E}}\end{array}$& $\begin{array}{l} \textrm{\textbf{Mixed}}\\\textrm{\textbf{- Fixed}}\end{array}$&$\hat{\bsn}'\hat{\bse}$  & $\begin{array}{l} \textrm{\textbf{Percentile}}\\\textrm{\textbf{in }}\hat{\bsn}'\bds{\pi}(\hat{\bse})\end{array}$&$\begin{array}{l} \textrm{\textbf{1k Sim}}\\\textrm{\textbf{Mean}}\end{array}$& $\begin{array}{l} \textrm{\textbf{1k Sim}}\\\textrm{\textbf{Bias}}\end{array}$ \\ \midrule
NCAA Football & 2.38 & 2.92&0.55 & 0.54 & 0.38 & $>99.9999$&2.86&0.49\\
NCAA Basketball (M) & 2.81 & 3.13&0.16 & 0.32 & 0.28 & $>99.9999$&3.14&0.33 \\
NCAA Basketball (W) & 2.94 & 3.10&0.17 & 0.16 & 0.14 & $>99.9999$&3.10&0.16 \\
NFL & 2.36 & 2.41&0.80 & 0.05 & 0.03 & 97.8&2.39&0.03 \\
NBA & 3.12 & 3.12&0.36 & 0.00 & 0.00 & 87.4&3.12&0.00 \\
WNBA & 3.26 & 3.26&0.86 & 0.00 & 0.00 & 42.3&3.27&0.01 \\ \bottomrule
\end{tabular}%
}
\end{table}

If every team in a sport plays the same number of home games as they play away games, then under model (\ref{eq:mixed}), $\bX'\bI\bZ=\bds{0}$ and the conditions of Corollaries~\ref{t:3} and \ref{t:6} hold. The NBA schedule is only slightly unbalanced in this sense due to neutral site games, while the WNBA schedule is completely balanced. This is reflected in Table~\ref{tab:2017sport} by the rounded-to-zero WNBA estimate for $\hat{\bsn}'\hat{\bse}$ (-3e-17) and by the  extremely small variability (due only to limitation in numerical precision) in the WNBA permutation distribution in Figure~\ref{fig:permsim}.

 To assess the practical magnitude of the internal diagnostic, the magnitude of $\hat{\bsn}'\hat{\bse}$ could be compared to the standard error of the corresponding point estimate. The standard errors reported in Table~\ref{tab:2017sport} are those produced by the software under the fixed-$\bZ$ assumption, rather than those that would be produced by Theorem~\ref{t:4}. For college men's and women's basketball, the diagnostic values are equal to $175\%$ and $82\%$, respectively, of the standard error, while for college football the diagnostic value is equal to $69\%$ of the standard error. In the professional sports, the diagnostic values are negligible compared to the standard errors.

\begin{figure}
\caption{: Randomized permutation distributions (one million permutations) of $\hat{\bsn}'\bds{\pi}(\hat{\bse})$ for the 2017 schedules for each sport. The solid line (blue) indicates the observed bias diagnostic, $\hat{\bsn}'\hat{\bse}$.  The dotted line (black) is placed at zero to show that the sampling distributions have mean 0 for this application. The two dashed lines (red) indicate the 0.5 and 99.5 percentiles for the permutation distributions.}
\label{fig:permsim}
\includegraphics[scale=.93]{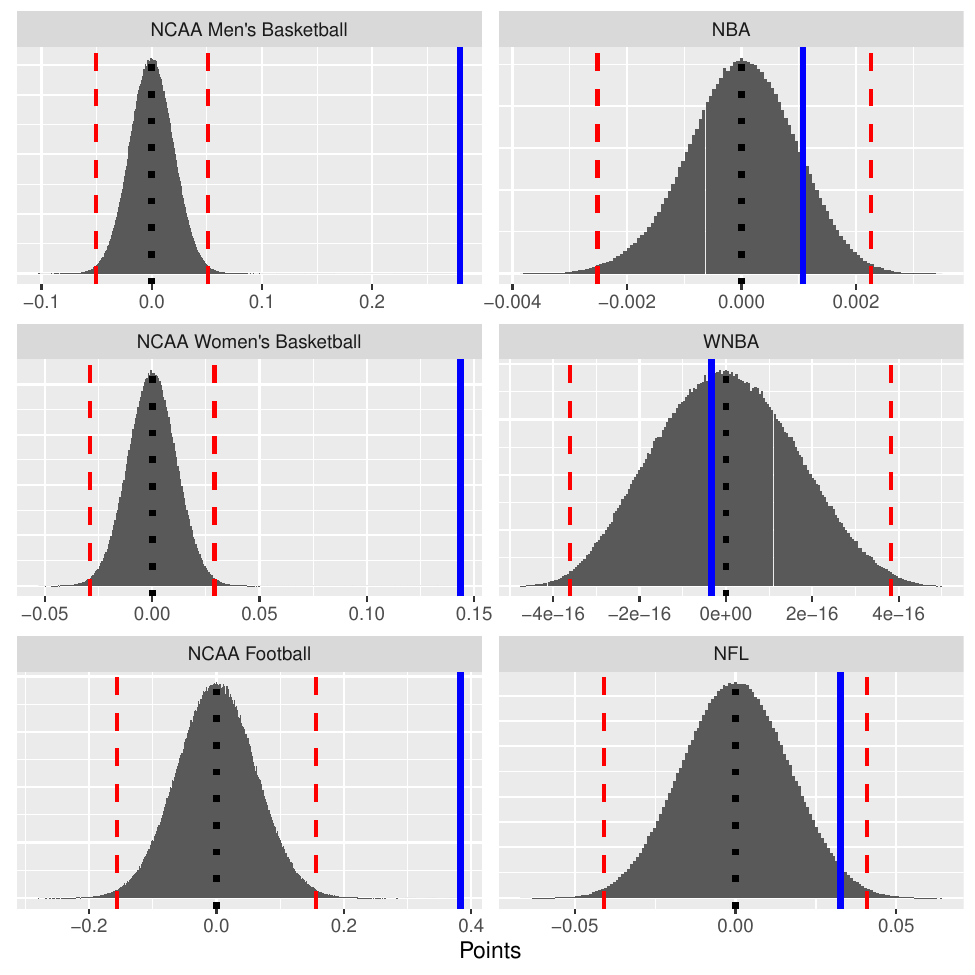}
\end{figure} 
 
\section{Simulating the Sports Scheduling Problem}\label{sec:sportssim}
The bias of $\bk'\hat{\hat{\bsB}}$ is $\textrm{E}[\hat{\bsn}'\bse]+C_k$. An experimenter with (approximate) knowledge of $\bse$ could select $\bZ$ to make the first term positive or negative; the sign of the total bias also depends on $C_k$.

Figure~\ref{fig:permsim} shows that the observed values of $\hat{\bsn}'\hat{\bse}$ for the college sports are larger than any of the million sampled permutations  $\hat{\bsn}'\bds{\pi}(\hat{\bse})$, suggesting that the nonrandom mechanism by which college sports schedules are built produces larger values of $\hat{\bsn}'\bse$ than would be expected under random scheduling. Using this knowledge, it is possible to reproduce the bias in the home field advantage estimator from the mixed model via simulation by selecting the schedule $\bZ$ (after $\bse$ has been generated) that maximizes $\bsn'\bse$, replicating the behavior of the real schedules observed in Figure~\ref{fig:permsim}. R code to reproduce the results of this section is available in the supplementary material \citep{kzdata}.

In order to simulate the biasing behavior seen in Figure~\ref{plot:marg1}, 5000 candidate schedules ($\bZ$) with 12 games per team are generated after the vector of 50 team strength effects, $\bse_s\sim N_{50}(\bds{0},225\bI)$, has been generated. The only restriction on the teams during the game assignments is that teams are not allowed to play themselves. As a result, the teams do not necessarily play the same number of home games.

The schedule, $\bZ_s$, that maximizes $\bsn'\bse_s$ is selected, and then game outcomes, $\bY_s=\bds{0}+\bZ_s\bse_s+\bds{e}_s$, are simulated, where $\bds{e}_s \sim N_{300}(\bds{0},529\bI)$. The variances were chosen to represent typical values observed from historical data. Note that $\bsn$ (with $k=1$ in this application) depends on $\bsG$ and $\bsR$; in practice, estimates of these matrices are available from previous seasons. $\bY_s$ is fit with the fixed effects model (Equation~\ref{eq:femodel}) and with the mixed effects model (Equation~\ref{eq:mixed}), and the estimates for the intercept are recorded. 1000 such simulations are run in this example, all using the same schedule ($\bZ_s$) and teams ($\bse_s$).

\begin{figure}
\caption{: Box plots of 1000 simulations of home field advantage estimates from both models when $\bsZ_s$ has been selected in order to maximize $\bsn'\bse$. The true value for the HFA in the simulation was 0. The third column contains the box plot for the matched pair differences between intercept estimates from each simulation: E-BLUE minus the BLUE.}
\label{fig:gamesimgraph1}
\includegraphics[scale=.55]{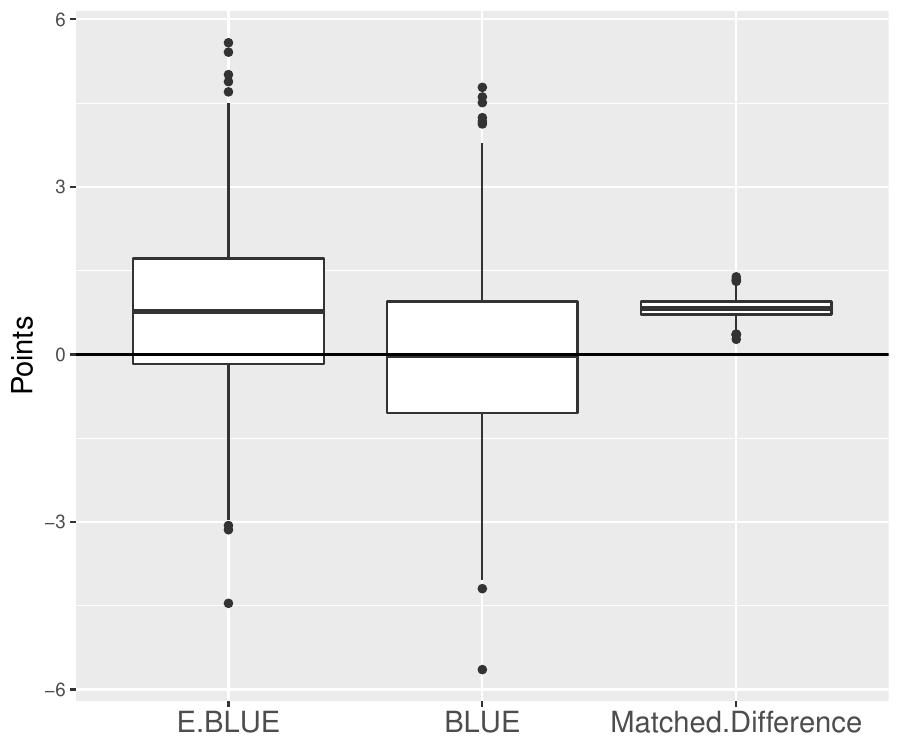}
\end{figure}

Figure~\ref{fig:gamesimgraph1} shows how the mixed effects model produces biased E-BLUEs of the intercept in the presence of this carefully chosen $\bZ_s$, while the BLUE from fitting the fixed effects model remains unbiased. The HFA E-BLUEs from the 1000 simulations have a mean of 0.7892 and a 95\% confidence interval of $(0.7014,0.8771)$. This confidence interval contains the known-covariance bias $\bsn'\bse=0.7108$. The corresponding BLUEs of HFA from fitting the fixed effects model have a mean of $-0.0440$ and a 95\% confidence interval of $(-0.1350, 0.0469)$, which contains the simulated HFA of 0. Figure~\ref{fig:gamesimgraph1} also shows the matched pair differences between the HFA estimates from the two models: the E-BLUE is uniformly larger than the BLUE from the fixed effects model for the same data set.  These differences show a correlation of $0.97$ with the observed values of $\hat{\bsn}'\hat{\bse}$.

\section{Conclusion}
We have seen how the mixed  model (Equation~\ref{eq:mixedmodel}) estimators of $\bk'\bsB$ can be biased when the random effects vector $\bse$ and the random effects model matrix $\bZ$ are dependent, unless  $\bZ$ is orthogonal to $\bX$ with respect to $\bR^{-1}$ (with probability 1).

The primary recommendations resulting from this study are
\begin{enumerate}
\item{Mixed model estimation software could print the approximate bias diagnostic, $\hat{\bsn}'\hat{\bse}$ (which omits $C_k$), next to the parameter estimates (or estimates of user-specified estimators), along with the percentile of this value from within the randomization distribution $\hat{\bsn}'\bds{\pi}(\hat{\bse})$, and the graphs shown in Figure~\ref{fig:permsim}. }
\item{When designing an experiment, choose $\bZ$ and $\bX$ such that $\bX'\bR^{-1}\bZ=\bds{0}$, for an anticipated error covariance matrix $\bR$, in order to obtain an unbiased estimator $\bk'\hat{\hat{\bsB}}$ and, under the assumptions of Theorem~\ref{t:4}, to minimize both $\Var(\bk'\hat{\bsB})$ and $\Var({\bsn}'\bse)$ for fixed $\bX$, $\bR$, and $\sigma^2$. For observational data analyses, it may be possible to include fixed effects (possibly covariates that were missing from an initial analysis) such that this orthogonality condition holds. }
\item{Be alert that, given an anticipated random effects vector $\bse$, it is possible to build $\bX$ and/or $\bZ$ to intentionally induce bias in the mixed model estimator $\bk
'\hat{\hat{\bsB}}$. This is, of course, not a recommended practice; however, there is a potential for this property to be exploited by a malicious party.}
\end{enumerate}

A similar biasing process also likely exists for generalized linear mixed models (GLMMs). The study is more difficult in that setting owing to a lack of closed form estimators of estimable functions of $\bsB$  \citep{karlcgs}. However, given a set of maximum likelihood estimates for the GLMM, the linearization technique of \citet{wolfinger93} -- the default estimation routine in SAS PROC GLIMMIX --  could be used in order to apply the methods of this paper to a  pseudo-response of the linearized GLMM likelihood function in order to obtain an approximation of the bias. Finally, the behavior of the predicted random effects themselves could also be studied. However, this situation is not as straightforward because these individual effects are not always estimable when converted to fixed effects (as with the HFA example) and because of the shrinkage properties of the random effects. For a predictor $\hat{\bse}$ of the random effects, unbiasedness requires $\bE[\hat{\bse}]=\bE[\bse]=\bds{0}$, and not that $\bE[\hat{\bse}|\bse]=\bse$ for all $\bse$ \citep{robinson1991}.

\section*{Appendix: Power Analysis for Randomization Test}
This appendix presents a simulation to study the power of the randomization test for independence of $\bsn$ and $\bse$ described in Section~\ref{sec:randomization}. This involves setting up a plausible simulation for a process under consideration and then examining the behavior of the randomization test as the simulated dependence between $\bsn$ and $\bse$ varies. In this case, we will use the 2017 Men's NCAA Basketball regular season results as a foundation, following the first two steps of the simulation procedure described in Section~\ref{sec:simtest}, except we will   simply use $\bds{0}$ instead of $\hat{\hat{\bsB}}$ in Step~2 in order to simulate a season with no HFA. We generate and fit two thousand new response vectors as described by that Step~2. For each of the resulting fitted models, we check whether the observed value $\hat{\bsn}'\hat{\bse}$ falls between the 2.5 and 97.5 percentiles of the randomized permutation distribution (with one million permutations) in order to form a hypothesis test with nominal level $\alpha=0.05$. The null hypothesis is that $\bsn$ and $\bse$ are independent, and it is rejected in favor of the alternative hypothesis of dependence when $\hat{\bsn}'\hat{\bse}$ falls outside of the interval. 

To see how the power of the test depends on the dependence of $\bsn$ and $\bse$, we modify the simulations by adding a switch proportion $p_s$ parameter that governs the proportion of the games from the 2017 schedule that switch home and away team assignments prior to each new simulated season. A proportion of 0 corresponds to the original schedule, and the rejection rate under this scenario yields an estimate of the power of the randomization test under the 2017 Men's NCAA scheduling scheme. A proportion of 0.5 corresponds to the original schedule game-pairings but with randomized home-away assignments. Intermediate values of the mixing proportion yield simulated schedules with different degrees of randomization for the home-away assignments. Finally, a scenario in which $\bse$ is shuffled prior to each simulated season (before being premultiplied by the same schedule structure, $\bZ$) yields a randomized schedule: the rejection rate here yields the observed Type I error rate that can be compared to the nominal rate of $\alpha=0.05$.

\begin{table}[]
\caption{: Results from 2000 simulations for each $p_s$ and ``shuffle schedule'' combination.  The rejection rate for the first row is the observed type I error rate, with a nominal probability of $\alpha=0.05$. The value of $\bsB=0$ was used to generate the simulations, meaning any nonzero values of Mean($\hat{\hat{\bsB}}$) represent bias. }
\label{tab:power}

\begin{tabular}{@{}lllll@{}}
\toprule
$p_s$ & $\begin{array}{l} \textrm{\textbf{Shuffle}}\\\textrm{\textbf{Schedule}}
\end{array}$ &$\begin{array}{l} \textrm{\textbf{Rejection}}\\\textrm{\textbf{Rate (RR)}}
\end{array}$  &$\begin{array}{l} \textrm{\textbf{Mean}}\\\hat{\hat{\beta}}
\end{array}$  &$\begin{array}{l} \textrm{\textbf{Mean}}\\\hat{\bsn}'\hat{\bse}
\end{array}$    \\ \midrule
0                 & TRUE       &   0.046    &   -0.0014       &-0.00046                  \\
0                 &  FALSE      &     1   &      0.33     &  0.28                 \\
0.25             &   FALSE      &   1   &     0.15      & 0.13            \\
0.40             &   FALSE      &   0.99    &     0.057      & 0.052                \\
0.45              &  FALSE      &   0.66     &     0.026      &   0.026               \\
0.50            &  FALSE       &    0.12   &     -0.0026      &   0.00034 \\
1.0               &   FALSE      &   1    &     -0.33      &    -0.28             \\ \bottomrule
\end{tabular}

\end{table}

Table~\ref{tab:power} shows the results of the simulation study. Even when the home-away assignments have been completely randomized ($p_s=0.5$), the rejection rate of 0.12 is larger than the nominal rate. However, when the game-pairing assignments themselves are randomized (Shuffle Schedule = TRUE), the rejection rate indicates a Type I error rate of 0.046. This suggests that there may be some additional non-random aspect of the schedule construction that is being detected in the $p_s=0.5$ scenario.  All of the simulations for the original schedule ($p_s=0$ and no shuffle) lead to a rejection of the null hypothesis.

The test has a power of 0.99 in detecting the positive bias even when 40\% of games have their home-away assignment switched in between each simulation. In this case, the magnitude of the bias has shrunk to 0.06 points from the bias of 0.33 points seen in the original schedule. This illustrates the importance of examining the practical significance of any detected bias. The scenario ($p_s=1$) considers switching every home and away assignment.  Unsurprisingly, this produces the same magnitude of bias in the HFA estimate, though in the opposite direction.

Since the HFA was set to zero in the simulation, the fourth column listing the mean of $\hat{\hat{\bsB}}$ from the simulations gives the bias in these scenarios. These values are strongly correlated with the  final column in the table, which lists the mean of the internal bias estimates ($\hat{\bsn}'\hat{\bse}$). This provides validation (within this particular example) of the use of the plug-in value $\hat{\bsn}'\hat{\bse}$ to estimate the bias of $\bk'\hat{\hat{\bsB}}.$

\section*{Acknowledgments}
Thank you to Sharon Lohr for reviewing and improving early drafts, including her suggestion to plot the randomization distributions in the manner of \citet[Figure 3.5]{boxhh}. We are also grateful for constructive comments from the anonymous referees.

\small
\bibliographystyle{agsm}

\bibliography{disbib4}
\end{document}